\documentclass{amsart} 
\usepackage[utf8]{inputenc}
\usepackage[T1]{fontenc}
\usepackage[top=2.54cm,bottom=2.54cm,right=3.5cm,left=3.5cm]{geometry}
\usepackage{color}
\usepackage{indentfirst}
\usepackage{lmodern} \normalfont
\DeclareFontShape{T1}{lmr}{bx}{sc} { <-> ssub * cmr/bx/sc }{}
\usepackage{amsmath}
\usepackage{amsfonts}
\usepackage{derivative}
\usepackage{amssymb}
\usepackage[alphabetic]{amsrefs}
\usepackage{color}
\usepackage{lipsum}                     
\usepackage{xargs}
\usepackage{comment}
\usepackage[pdftex,dvipsnames]{xcolor} 
\usepackage{authblk}
\usepackage{comment}

\usepackage{amsthm,mathrsfs}
\usepackage{slashed}
\usepackage{breqn}
\usepackage{pb-diagram}
\usepackage{bbm}
\usepackage[all]{xy}
\usepackage{times}
\usepackage{amsaddr}
\usepackage{microtype}
\usepackage{enumerate}
\usepackage[super]{nth}
\usepackage{multicol}
\usepackage{tikz-cd}
\usepackage{array}
\usepackage[colorinlistoftodos]{todonotes}
\usepackage{slashed}

\usepackage{hyperref}

\renewcommand{\epsilon}{\varepsilon}

\def\<{\mathopen{}\left<}
\def\>{\right>\mathclose{}}
\def\({\mathopen{}\left(}
\def\){\right)\mathclose{}}

\usepackage{multicol, color}

\definecolor{gold}{rgb}{0.85,.66,0}
\definecolor{cherry}{rgb}{0.9,.1,.2}
\definecolor{burgundy}{rgb}{0.8,.2,.2}
\definecolor{orangered}{rgb}{0.85,.3,0}
\definecolor{orange}{rgb}{0.85,.4,0}
\definecolor{olive}{rgb}{.45,.4,0}
\definecolor{lime}{rgb}{.6,.9,0}
\definecolor{green}{rgb}{.2,.7,0}
\definecolor{grey}{rgb}{.4,.4,.2}
\definecolor{brown}{rgb}{.4,.3,.1}

\makeatletter
\newtheorem*{rep@theorem}{\rep@title}
\newcommand{\newreptheorem}[2]{%
\newenvironment{rep#1}[1]{%
 \def\rep@title{#2 \ref{##1}}%
 \begin{rep@theorem}}%
 {\end{rep@theorem}}}
\makeatother

\newreptheorem{theorem}{Theorem}
\newtheorem{lemma}{Lemma}[section]
\newtheorem{corollary}[lemma]{Corollary}
\newtheorem{definition}[lemma]{Definition}

\newtheorem{remark}[lemma]{Remark}

\newtheorem*{remark*}{Remark}
\newtheorem*{theorem*}{Theorem}
\newtheorem{proposition}[lemma]{Proposition}

\newcommandx{\andresnote}[2][1=]{\todo[linecolor=blue,backgroundcolor=blue!25,bordercolor=blue,#1]{#2}}

\newcommandx{\Julinote}[2][1=]{\todo[linecolor=red,backgroundcolor=yellow!25,bordercolor=red,#1]{#2}}

\title{Contact Whirl Curves in Sasakian Lorentzian $3$-Manifolds}

\author{Luis E. Portilla P.}
\address{Department of Mathematics, University of Cauca, Co.}
\email{luisportilla@unicauca.edu.co}

\author{Hector  E. Guerrero M.}
\address{Department of Mathematics, University of Cauca, Co.}
\email{hguerrero@unicauca.edu.co}

\date{\today}
\begin{document}
 
\maketitle
\begin{abstract}
We introduce and study \emph{contact whirl curves} in three-dimensional Lorentzian contact manifolds, with emphasis on the Sasakian setting. This notion refines the concept of whirl curves by encoding the interaction between the adapted frame of a curve  and the ambient contact structure through the Reeb vector field.
For non-geodesic unit-speed contact whirl curves, we derive a differential equation governing the torsion in terms of the Frenet invariants and the contact data. In the Lorentzian Sasakian setting, this leads to rigidity phenomena of Lancret type. In particular, we prove that every non-geodesic Legendre Frenet curve is automatically a contact whirl curve, and consequently has constant torsion $\tau=1$.

We also investigate the interaction between contact whirl curves and magnetic trajectories associated with the canonical contact magnetic field. We show that every non-geodesic curve which is simultaneously magnetic and contact whirl must be Legendre, and we obtain an explicit expression for its torsion in terms of the tensor $h=\frac12\mathcal L_\xi\Phi$. In the Sasakian case, this reduces to the universal law $\tau=1$. Finally, in the Lorentzian Heisenberg group endowed with its standard Sasakian structure, we derive a coordinate form of the whirl condition and use it to produce explicit examples, including a construction by quadratures of non-Legendre contact whirl curves and a horizontal helicoidal Legendre family.
\end{abstract}
%
\section{Introduction}
The geometry of special types of curves in contact manifolds has been widely studied from several viewpoints, including Frenet theory, the interaction with the Reeb field, and the dynamics induced by magnetic trajectories. In Lorentzian contact geometry, this  includes the theory of slant curves, Legendre curves, null curves, and contact magnetic curves   \cite{cho2006slant,lee2019slant,blair2010contact}. Motivated by these developments, in this paper we introduce a new family of curves in Lorentzian Sasakian $3$--manifolds, which we call \emph{contact whirl curves}.\\ 

More precisely, motivated by \cite{guerrero2025lancret}, let $\alpha$ be a regular curve in a Lorentzian Sasakian $3$--manifold $(M,\eta,\xi,g,\Phi)$, and assume that along $\alpha$ there exists an adapted frame $\{T,N,B\}$, namely, the Frenet frame in the non-null case or the Cartan frame in the non-geodesic null case. We say that $\alpha$ is a \textit{contact whirl curve} if there exists a constant $\varrho\neq 0$ such that
$$
\eta(T)=\varrho\,\eta(N).
$$
This whirl condition couples the Reeb components of the tangent and normal directions, and therefore naturally relates the geometry of the adapted frame to the ambient Sasakian structure.\\

The first part of the paper develops the basic structure of this notion. In Section~\ref{Sec:2} we recall the necessary background on Lorentzian contact metric and Sasakian $3$--manifolds, including the Frenet formalism for non-null curves and the identities satisfied by the structure tensors and we introduce the notion of contact whirl curves. In Section~\ref{Sec:3}  we  derive the fundamental differential identities satisfied by the Reeb components $\eta(T)$, $\eta(N)$, and $\eta(B)$. These identities lead to a local differential equation (cf. Proposition \ref{prop:tauprime_whirl}) governing the torsion of a non-geodesic contact whirl curve in the Lorentzian Sasakian setting. In particular, under explicit nondegeneracy assumptions, we prove that the torsion satisfies a Lancret-type differential relation, which can also be rewritten as an evolution equation for the scalar quantity
\[
\frac{\tau-1}{\kappa}.
\]
This provides a useful description of the coupled behavior of curvature and torsion for non-geodesic whirl curves. 

We analyze two distinguished subclasses. The first one is the Legendre case, we prove that a Legendre Frenet curve is contact whirl if and only if $\eta(N)=0$, and that every non-geodesic Legendre contact whirl curve necessarily has torsion $\tau=1.$ (cf. Proposition \ref{prop:legendreWhirl}). Conversely, in a Lorentzian Sasakian $3$--manifold every non-geodesic Legendre Frenet curve is automatically a contact whirl curve. Thus, within this class, the contact whirl condition is both natural and rigid. 

The second subclass arises from magnetic trajectories associated with the canonical contact magnetic field. In Section~\ref{sec:magnetic} we show that if a non-geodesic Frenet curve is simultaneously magnetic and contact whirl, then it must be Legendre (cf. Proposition \ref{prop:magnetic_legendre}). Moreover, its torsion is explicitly determined by the tensor $h=\frac12\mathcal L_\xi\Phi$. Furthermore, in the Sasakian case one recovers the universal law $\tau=1.$ This gives a dynamical interpretation of the whirl condition and provides an analogue of the  compatibility phenomena observed by Lee for slant magnetic curves.\\

In Section~\ref{Sec:4} we consider the null case. We work directly with the Cartan frame and the Sasakian identities. We obtain a structural description of a hypothetical non-geodesic null contact whirl curve: the Reeb components, the expression of $\Phi T$, the torsion, and the decomposition of the Reeb field relative to the Cartan frame are all shown to be strongly constrained. In this sense, Section~\ref{Sec:4} provides the null analogue of the structural results known for null slant curves \cite{lee2019slant}.\\

Finally, in Section~\ref{Sec:5} we realize the theory in the Lorentzian Heisenberg group. We derive a coordinate form of the whirl condition and show that   it can be used as a construction principle. More precisely, the whirl condition  reduces, in this model, to a first-order equation for the Reeb component of the tangent field, and this yields a method of construction by quadratures for non-Legendre contact whirl curves. We also exhibit an explicit horizontal helicoidal Legendre family, thereby showing that the class of contact whirl curves is nonempty and geometrically rich even in this homogeneous Sasakian model. Taken together, these results show that contact whirl curves form a natural and robust class of curves in Lorentzian Sasakian geometry. They interact in a nontrivial way with Legendre geometry, magnetic dynamics, null Cartan theory, and explicit left-invariant models such as the Lorentzian Heisenberg group.\\

\medskip

\noindent\textbf{Acknowledgements}
The authors are grateful to the Geometry Seminar of the Research Group on Functional Spaces at the University of Cauca for providing a stimulating academic environment in which part of this work was developed.
\section{Preliminaries}\label{Sec:2}
Let $(M,\eta,\xi,g,\Phi)$ be a $3$--dimensional Lorentzian contact metric manifold, where $g$ is a Lorentzian metric, $\eta$ is the contact form, $\xi$ is the Reeb vector field, and $\Phi$ is the structure endomorphism. We adopt the convention
\[
\eta(\xi)=1 \quad\text{and}\quad g(\xi,\xi)=-1.
\]
The structure tensors satisfy
\begin{equation}\label{eq:gdeta}
g(\Phi X,\Phi Y)=g(X,Y)+\eta(X)\eta(Y),
\end{equation}
for all vector fields $X,Y$ on $M$, and the associated fundamental $2$--form is given by 
$$
\omega(X,Y) =d\eta(X,Y) =g(X,\Phi Y).
$$ 
Moreover, taking $Y=\xi$ in \eqref{eq:gdeta}, we obtain
\begin{equation}\label{eq:eta}
\eta(X)=-g(X,\xi).
\end{equation}
In a general contact Lorentzian manifold one has
\begin{equation}\label{eq:Phi 1}
\nabla_X\xi=\Phi X-\Phi hX,\quad\text{where}\quad h=\frac12\mathcal L_\xi\Phi.
\end{equation} 
A contact Lorentzian manifold is called \emph{$K$--contact} if the Reeb vector field $\xi$ is Killing. In dimension $3$, this is equivalent to $h=0$, i.e., $(M,\eta,\xi,g,\Phi)$ is  Sasakian \cite{blair2006contact}. Therefore \eqref{eq:Phi 1} reduces to
\begin{equation}\label{eq:Phi 2}
\nabla_X\xi=\Phi X.
\end{equation}
Hence, throughout this paper, the assumption that $\xi$ is Killing guarantees the full Sasakian structure and in this case we have 
\begin{equation}\label{eq:sasakianlorentz}
(\nabla_X\Phi)Y=g(X,Y)\xi+\eta(Y)X.
\end{equation}
Let $\alpha:I\to M$ be a unit--speed non-null curve such that
\[
g(\alpha',\alpha')=\varepsilon_1\in\{1,-1\}.
\]
The constant $\varepsilon_1$ is called the causal character of $\alpha$.
Accordingly, $\alpha$ is said to be spacelike if $\varepsilon_1=1$ and timelike if $\varepsilon_1=-1$.
Assume that $\nabla_{\alpha'}\alpha'$ is non-null. Then $\alpha$ admits an orthonormal Frenet frame $\{T=\alpha',N,B\}$,   satisfying
\[
g(T,T)=\varepsilon_1,\quad
g(N,N)=\varepsilon_2,\quad
g(B,B)=\varepsilon_3, \quad\text{with } \quad\varepsilon_1\varepsilon_2=-\varepsilon_3. 
\]
The Frenet equations are
\begin{equation}\label{eq:FrenetSerret}
\begin{cases}
\nabla_{\alpha'}T &=\varepsilon_2\kappa N,\\ 
\nabla_{\alpha'}N &=-\varepsilon_1\kappa T-\varepsilon_3\tau B,\\ 
\nabla_{\alpha'}B &=\varepsilon_2\tau N,
\end{cases}
\end{equation}
where $\kappa>0$ is the curvature and $\tau$ is the torsion of $\alpha$. 
For orthonormal non-degenerate frames, we shall use  the Lorentzian vector product introduced in \cite{camci2012extended}:
\begin{equation}\label{eq: cross}
X\wedge Y=g(X,\Phi Y)\xi-\eta(Y)\Phi X+\eta(X)\Phi Y.
\end{equation}
This product is bilinear, skew-symmetric, and orthogonal to both $X$ and $Y$. Taking $Y=\xi$ in \eqref{eq: cross}, we obtain
\begin{equation}\label{eq:xiwedgex}
\Phi X=\nabla_X\xi=-X\wedge\xi,
\end{equation}
Finally, if $\{T,N,B\}$ is an orthonormal frame in a Lorentzian $3$--manifold, then
\begin{equation}\label{eq:CrossBasis}
T\wedge N=\varepsilon_3 B,\qquad
N\wedge B=\varepsilon_1 T,\qquad
B\wedge T=\varepsilon_2 N.
\end{equation}
\subsection{Contact whirl curves}
Now we introduce  a new family of  curves in Lorentzian Sasakian $3$--manifolds, motivated by the theory of slant curves developed in \cite{lee2019slant,cho2006slant}. This is the central notion to work with in this paper. 
\begin{definition}
Let $\alpha:I\to M$ be a regular curve in a Lorentzian Sasakian $3$--manifold.
Assume that along $\alpha$ there exists an adapted frame $\{T,N,B\}$, namely, the Frenet frame in the non-null case, or the Cartan frame in the non-geodesic null case. We say that $\alpha$ is a contact whirl curve if there exists a nonzero constant $\varrho\in\mathbb R$ such that
\begin{equation}\label{eq:whirlcurve}
\eta(T)=\varrho\,\eta(N).
\end{equation}
\end{definition}
The condition \eqref{eq:whirlcurve} should be compared with the slant condition \cite{lee2019slant}, where the angle between the tangent  vector and the Reeb field is constant. In contrast, a contact whirl curve couples the Reeb components of the tangent and normal directions. 
\section{Frenet Contact Whirl Curves}\label{Sec:3}
In this section we restrict to the non-null case, so that the adapted frame is the Frenet frame. 
\subsection{Basic identities and the torsion equation}
We next collect the basic identities that will be used throughout the paper. Since $\{T,N,B\}$ is an orthonormal frame along $\alpha$, the Reeb vector field decomposes as
\begin{equation}\label{eq:reeb1}
\begin{aligned}
\xi
&=\varepsilon_1 g(\xi,T)T+\varepsilon_2 g(\xi,N)N+\varepsilon_3 g(\xi,B)B\\
&=-\varepsilon_1\eta(T)T-\varepsilon_2\eta(N)N-\varepsilon_3\eta(B)B.
\end{aligned}
\end{equation}
We shall use the notation $\varepsilon_{ij}:=\varepsilon_i\varepsilon_j,$ and $\varepsilon_{ijk}:=\varepsilon_i\varepsilon_j\varepsilon_k.$
Using \eqref{eq:xiwedgex}, \eqref{eq:CrossBasis}, and the decomposition \eqref{eq:reeb1}, the action of the tensor $\Phi$ on the Frenet frame can be written explicitly. Indeed,
\begin{align*}
\Phi T &=\xi\wedge T =\big(-\varepsilon_2\eta(N)N-\varepsilon_3\eta(B)B\big)\wedge T 
=\varepsilon_{23}\eta(N)B-\varepsilon_{23}\eta(B)N.
\end{align*}
Proceeding in the same way for $N$ and $B$, we obtain
\begin{equation}\label{eq:PhiEi}
\begin{aligned}
\Phi T &= \varepsilon_{23}\big(\eta(N)B-\eta(B)N\big),\\
\Phi N &= \varepsilon_{13}\big(\eta(B)T-\eta(T)B\big),\\ 
\Phi B &= \varepsilon_{12}\big(\eta(T)N-\eta(N)T\big).
\end{aligned}
\end{equation}
Differentiating the functions $\eta(T)$, $\eta(N)$ and $\eta(B)$ along the curve, and using \eqref{eq:eta}, \eqref{eq:FrenetSerret}, \eqref{eq:Phi 1}, and \eqref{eq:PhiEi}, we obtain the identities that will be used later.  First,
\begin{align*}
\eta(T)'&=-\frac{d}{ds}g(T,\xi) 
=-g(\nabla_TT,\xi)-g(T,\nabla_T\xi)\\
&=-g(\varepsilon_2\kappa N,\xi)-g\bigl(T,\Phi T-\Phi hT\bigr)
=\varepsilon_2\kappa\,\eta(N)+g(T,\Phi hT),
\end{align*}
and therefore
\begin{equation}\label{eq:etaprima1}
\eta(T)'=\varepsilon_2\kappa\,\eta(N)+g(T,\Phi hT).
\end{equation}
Similarly,
\begin{align*}
\eta(N)'&=-\frac{d}{ds}g(N,\xi) =-g(\nabla_TN,\xi)-g(N,\nabla_T\xi)\\
&=-g\bigl(-\varepsilon_1\kappa T-\varepsilon_3\tau B,\xi\bigr)-g\bigl(N,\Phi T-\Phi hT\bigr),
\end{align*}
which gives
\begin{equation}\label{eq:etaprima2}
\eta(N)'=-\varepsilon_1\kappa\,\eta(T)-\varepsilon_3(\tau-1)\eta(B)+g(\Phi hT,N).
\end{equation}
Finally,
\begin{align*}
\eta(B)' &=-\frac{d}{ds}g(B,\xi)=-g(\nabla_TB,\xi)-g(B,\nabla_T\xi)\\
&=-g(\varepsilon_2\tau N,\xi)-g\bigl(B,\Phi T-\Phi hT\bigr),
\end{align*}
and hence
\begin{equation}\label{eq:etaprima3}
\eta(B)'=\varepsilon_2(\tau-1)\eta(N)+g(B,\Phi hT).
\end{equation}

In particular, when the Reeb vector field $\xi$ is Killing, one has $h=0$, and the above identities reduce to the Sasakian case. The following proposition gives a local differential relation governing the torsion of a non-geodesic contact whirl curve in the Lorentzian Sasakian setting.
\begin{proposition}\label{prop:tauprime_whirl}
Let $(M^{3},\Phi,\xi,\eta,g)$ be a Lorentzian Sasakian manifold and
let $\alpha:I\to M$ be a non-geodesic contact whirl curve with Frenet frame
$\{T,N,B\}$.
Assume that, on an open interval $J\subset I$, one has $\eta(T)\neq 0,$ $\tau\neq 1,$ and  $A:=\varepsilon_1+\frac{\varepsilon_2}{\varrho^2}\neq 0.$
Then the torsion $\tau$ satisfies on $J$
\begin{equation}\label{eq:tauprime_whirl}
\tau'
=
(\tau-1)\!\left(
\frac{\varepsilon_2\varepsilon_3}{\varrho\,\kappa\,A}(\tau-1)^2
+\frac{\kappa'}{\kappa}
+\frac{\varepsilon_2}{\varrho}\kappa
\right).
\end{equation}
\end{proposition}
\begin{proof} 
Using the contact whirl condition $\eta(N)=\frac{1}{\varrho}\eta(T)$, we first rewrite \eqref{eq:etaprima1}--\eqref{eq:etaprima3} eliminating $\eta(N)$.  From \eqref{eq:etaprima1} we obtain
$$
\eta(T)'=\frac{\varepsilon_2}{\varrho}\,\kappa\,\eta(T).
$$
Differentiating the Whirl condition \eqref{eq:whirlcurve}  gives:
$\eta(N)'=\frac{1}{\varrho}\eta(T)'=\frac{\varepsilon_2}{\varrho^2}\,\kappa\,\eta(T).$
Comparing this with \eqref{eq:etaprima2} we get an expression for $\eta(B )$:
\[
\frac{\varepsilon_2}{\varrho^2}\kappa\eta(T) = -\varepsilon_1\kappa\eta(T)
-\varepsilon_3(\tau-1)\eta(B ),
\]
hence
\begin{equation}\label{eq:eta3express}
\eta(B ) = -\frac{\kappa\,( \varepsilon_1 + \frac{\varepsilon_2}{\varrho^2})}{\varepsilon_3(\tau-1)}\,\eta(T).
\end{equation}
On the other hand \eqref{eq:etaprima3} and contact whirl condition yields
\[
\eta(B )'=\varepsilon_2(\tau-1)\eta(N)=\frac{\varepsilon_2}{\varrho}(\tau-1)\eta(T).
\]
Differentiate the right-hand side of \eqref{eq:eta3express} and comparing  with the last equality.
Write $A:=\varepsilon_1+\frac{\varepsilon_2}{\varrho^2}$ and $D:=\varepsilon_3(\tau-1)$. Then
\[
\eta(B )=-\frac{\kappa A}{D}\eta(T).
\]
Differentiating gives
\[
\eta(B )'=-\frac{A}{D}\kappa'\eta(T)
+\frac{\kappa A}{D^2}D'\,\eta(T)
-\frac{\kappa A}{D}\eta(T)',
\]
and since $D'=\varepsilon_3\tau'$ and $\eta(T)'=\frac{\varepsilon_2}{\varrho}\kappa\eta(T)$ this becomes
\[
\eta(B )' = -\frac{A}{\varepsilon_3(\tau-1)}\kappa'\eta(T)
+\frac{\kappa A}{\varepsilon_3(\tau-1)^2}\tau'\,\eta(T)
-\frac{\kappa A}{\varepsilon_3(\tau-1)}\frac{\varepsilon_2}{\varrho}\kappa\eta(T).
\]
Equating this with $\eta(B )'=\frac{\varepsilon_2}{\varrho}(\tau-1)\eta(T)$ and cancelling the common factor $\eta(T)$ yields
\[
\frac{\varepsilon_2}{\varrho}(\tau-1)
= -\frac{A}{\varepsilon_3(\tau-1)}\kappa'
+ \frac{\kappa A}{\varepsilon_3(\tau-1)^2}\tau'
- \frac{\kappa A}{\varepsilon_3(\tau-1)}\frac{\varepsilon_2}{\varrho}\kappa .
\]
Multiplying by $\varepsilon_3(\tau-1)^2$ and isolating $\tau'$ gives the   ODE \eqref{eq:tauprime_whirl} for $\tau$.
\end{proof}
\begin{remark}
In the non-geodesic case Proposition \ref{prop:tauprime_whirl} recovers, in the non-geodesic case, the analogous differential relation obtained in \cite{guerrero2025lancret}. On any interval where the hypotheses of Proposition \ref{prop:tauprime_whirl} hold, equation \eqref{eq:tauprime_whirl} can be rewritten as
\begin{equation}\label{eq:u_whirl}
\left(\frac{\tau-1}{\kappa}\right)'=
\varepsilon_2\varrho\,(\tau-1)
\left(
1-\frac{1}{\varepsilon_2+\varepsilon_1\varrho^2}
\left(\frac{\tau-1}{\kappa}\right)^2
\right).
\end{equation}
This form isolates the coupled behavior of curvature and torsion through the scalar quantity $\frac{\tau-1}{\kappa}$, and is convenient for qualitative analysis of the corresponding differential system.
The branch $\tau\equiv1$ corresponds to a singular solution of \eqref{eq:tauprime_whirl}.
\end{remark}
We now turn to distinguished subclasses of non-geodesic contact whirl curves, namely $\kappa>0$, for which the Frenet frame is well defined.   
\subsection{Legendre contact whirl curves}
We now study the interaction between the Legendre condition and the contact whirl relation in Lorentzian Sasakian $3$--manifolds.

\begin{definition}\label{def:legendre}
Let $(M,\eta,\xi,g,\Phi)$ be a contact metric Lorentzian $3$--manifold. A smooth  regular curve $\alpha\colon I\to M$ is called a \emph{Legendre curve} if its tangent vector field lies in the contact distribution, that is,  $T\in \ker\eta$ along $\alpha$.
\end{definition}
Hence, if $\alpha$ is a Legendre curve, then $\eta(T)=0$, and from \eqref{eq:whirlcurve} we immediately obtain the following  characterization.
\begin{proposition}\label{prop:Legendre}
Let $\alpha$ be a Legendre Frenet curve in a Lorentzian Sasakian $3$--manifold. Then $\alpha$ is a contact whirl curve if and only if $\eta(N)=0.$
\end{proposition}
Assume now that $\alpha$ is a non-geodesic Legendre contact whirl curve. Then Proposition \ref{prop:Legendre} gives $\eta(N)=0$. From \eqref{eq:etaprima3}, it follows that $\eta(B)$ is constant. Substituting into \eqref{eq:etaprima2}, we obtain
\[
(\tau-1)\eta(B)=0.
\]
Since the Reeb vector field has unit length and is orthogonal to both $T$ and $N$, one has $\eta(B)\neq0$. Therefore, $\tau=1.$ This proves the following
\begin{proposition}\label{prop:legendreWhirl}
Let $\alpha$ be a non-geodesic Legendre contact whirl curve in a Lorentzian Sasakian $3$--manifold. Then its torsion is $\tau=1.$
\end{proposition}
This agrees with the classical behavior of Legendre curves in Sasakian $3$--geometry; compare with Proposition~8.2 of \cite{blair2010contact}.\\

Conversely, Proposition~8.13 of \cite{blair2010contact} states that, on a contact $3$--manifold, if a Legendre curve has torsion $\tau=1$, then the ambient manifold is Sasakian. Hence, within the class of non-geodesic Legendre contact whirl curves in dimension three, the condition $\tau=1$ is equivalent to the ambient manifold being Sasakian.
Finally, equation \eqref{eq:etaprima1} shows that every non-geodesic Legendre curve satisfies $\eta(N)=0$. Therefore:
\begin{proposition}
Let $M$ be a Lorentzian Sasakian $3$--manifold. Every non-geodesic Legendre Frenet curve is a contact whirl curve.
\end{proposition} 
%
%
\subsection{Magnetic trajectories and contact whirl curves}\label{sec:magnetic}
%
The interaction between magnetic trajectories and special classes of curves has proved fruitful in Lorentzian contact geometry. In the slant setting, Lee \cite{lee2019slant} showed that contact magnetic curves are slant if and only if the ambient contact Lorentzian $3$--manifold is Sasakian. Motivated by this characterization, we study here the analogous compatibility problem for contact whirl curves.

\begin{definition}\label{def:magneticField}
A magnetic field on a semi-Riemannian manifold $(M,g)$ is a closed $2$--form $F$. Its associated Lorentz force is the $(1,1)$--tensor field $\varphi$ determined by
$$
g(\varphi(X),Y)=F(X,Y),
\qquad X,Y\in TM.
$$
A magnetic trajectory is a curve $\alpha$ satisfying the Lorentz equation
\begin{equation}\label{eq:lorentz}
\nabla_{\dot\alpha}\dot\alpha=\varphi(\dot\alpha).
\end{equation}
\end{definition}

\begin{definition}\label{def:lorentzforce}
Let $(M,\eta,\xi,g,\Phi)$ be a contact Lorentzian $3$--manifold and let $q\neq0$. The contact magnetic field of strength $q$ is the closed $2$--form
$$
F_{\xi,q}(X,Y)=-q\,d\eta(X,Y).
$$
\end{definition}

The corresponding Lorentz force is given by
$ 
\varphi=q\Phi.
$ 
Hence, if $T=\dot\alpha$, the magnetic equation becomes
\begin{equation}\label{eq:contact_lorentz}
\nabla_TT=q\Phi T.
\end{equation}
We first derive a  identity valid on any contact Lorentzian $3$--manifold. Differentiating $\eta(T)=-g(T,\xi)$ along a magnetic trajectory and using \eqref{eq:contact_lorentz} together with $ \nabla_X\xi=\Phi X-\Phi hX$ \eqref{eq:Phi 1},   we obtain
\begin{equation}\label{eq:magnetic_eta_general}
\begin{aligned}
\eta(T)'
&=-g(\nabla_TT,\xi)-g(T,\nabla_T\xi)
=-q\,g(\Phi T,\xi)-g(T,\Phi T)+g(T,\Phi hT)\\
&=g(T,\Phi hT).
\end{aligned}
\end{equation}
Thus the variation of the Reeb component along a magnetic trajectory is completely governed by the tensor $h$. We now combine the magnetic condition with the whirl constraint.
\begin{proposition}\label{prop:magnetic_legendre}
Let $\alpha\colon I\to M$ be a non-geodesic Frenet curve in a contact Lorentzian $3$--manifold. If $\alpha$ is both a magnetic trajectory and a contact whirl curve, then $\alpha$ is necessarily Legendre. In particular,
$\eta(T)=0$ and $\eta(N)=0.$
\end{proposition}
\begin{proof}
Assume that $\alpha$ satisfies the whirl condition \eqref{eq:whirlcurve}
$\eta(T)=\varrho\,\eta(N),$   $\varrho\neq0.$ Since $\alpha$ is non-geodesic, $\kappa\neq0$. By \eqref{eq:etaprima1},
$$
\eta(T)' =\varepsilon_2\kappa\,\eta(N)+g(T,\Phi hT) 
$$
On the other hand, equation \eqref{eq:magnetic_eta_general} gives
$$
\eta(T)'=g(T,\Phi hT).
$$
Comparing both expressions for $\eta(T)'$, we obtain
$$
\varepsilon_2\kappa\,\eta(N)=0.
$$
Since $\kappa\neq0$ and $\varepsilon_2=\pm1$, it follows that $ \eta(N)=0.$
Since $\varrho\neq0$, the whirl condition implies
$ \eta(T)=0. $  Thus $\alpha$ is Legendre.
\end{proof}
We next derive the corresponding torsion law.
\begin{proposition}\label{prop:magnetic}
Let $\alpha\colon I\to M$ be a non-geodesic magnetic whirl curve in a contact Lorentzian $3$--manifold. Then its torsion satisfies
\begin{equation}\label{eq:magnetic}
\tau=1-\varepsilon_1\,g(hT,T).    
\end{equation} 
\end{proposition}
\begin{proof}
By Proposition~\ref{prop:magnetic_legendre}, we have
$\eta(T)=0=\eta(N)$. Hence the Reeb vector field is aligned with the binormal direction, so $\eta(B)^2=1$ and $\eta(B)$ is constant. Using  \eqref{eq:etaprima2}, we obtain
$$
0=-\varepsilon_1\kappa\,\eta(T)-\varepsilon_3(\tau-1)\eta(B)+g(\Phi hT,N).
$$
Since $\eta(T)=0$, this reduces to
$$
0=-\varepsilon_3(\tau-1)\eta(B)+g(\Phi hT,N).
$$
Now, from \eqref{eq:PhiEi} and $\eta(T)=\eta(N)=0$, we have
$$
\Phi N=\varepsilon_{13}\eta(B)\,T.
$$
Using the skew-symmetry of $\Phi$,
$$
g(\Phi hT,N)=-g(hT,\Phi N) =-\varepsilon_{13}\eta(B)\,g(hT,T).
$$
Since $\xi=-\varepsilon_3\eta(B)B$ and $g(\xi,\xi)=-1$, we obtain
$$
\varepsilon_3\eta(B)^2=-1.
$$
Hence $\varepsilon_3=-1$ and $\eta(B)^2=1$, i.e., $\eta(B)\neq0$,  it follows that
$$
(\tau-1)\eta(B)+\varepsilon_1\eta(B)\,g(hT,T)=0.
$$
dividing by $\eta(B)$ yields
$ 
\tau=1-\varepsilon_1 g(hT,T),
$ 
which proves \eqref{eq:magnetic}.
\end{proof}
In the Sasakian case, then $h=0$. From Proposition~\ref{prop:magnetic} follows immediately the following 
\begin{corollary}
Let $\alpha\colon I\to M$ be a non-geodesic magnetic whirl curve in a Lorentzian Sasakian $3$--manifold. Then the torsion is $ \tau=1. $
\end{corollary} 
Therefore, magnetic trajectories provide a natural dynamical setting in which the whirl condition forces the Legendre condition  and determines the torsion. In the Sasakian case, this rigidity reduces to the universal value $\tau=1$, in parallel with the slant magnetic characterization obtained in \cite{lee2019slant}.
%
\section{Null Contact Whirl Curves}\label{Sec:4}
In the null case, the Cartan frame is degenerate, so the vector product formalism used in the non-null setting is no longer available. Following the strategy used for null slant curves in \cite{lee2019slant}, we work intrinsically with the Cartan frame and the Sasakian identities.\\

Let $\alpha$ be a null curve in a Lorentzian Sasakian $3$--manifold $(M,\eta,\xi,g,\Phi)$, parametrized by pseudo arc-length. We assume that $\alpha$ is non-geodesic, so that there exists a Cartan frame $\{T,N,B\}$ satisfying
\begin{equation}\label{eq:cartan_metric_nonexist}
\begin{aligned}
g(T,T)=g(B,B)=g(T,N)=g(B,N)=0,\qquad
g(T,B)=1=g(N,N),
\end{aligned}
\end{equation}
and
\begin{equation}\label{eq:cartan_frenet_nonexist}
\nabla_TT=N,\qquad
\nabla_TN=-\tau T-B,\qquad
\nabla_TB=\tau N.
\end{equation}
Assume that $\alpha$ satisfies the contact whirl condition  \eqref{eq:whirlcurve} $\eta(T)=\varrho\,\eta(N),$ where $\varrho\neq0.$ It is convenient to write $\rho:=\frac{1}{\varrho}.$  For some smooth functions $x,y,z$, we decompose the Reeb vector field along $\alpha$ as
\[
\xi=xT+yN+zB.
\]
From \eqref{eq:cartan_metric_nonexist}, we have $\eta(T)=-g(T,\xi)=-z,$
and $\eta(N)=-g(N,\xi)=-y.$  Hence the whirl condition becomes $z=\varrho y,$ equivalently $y=\rho z.$ Moreover, since $g(\xi,\xi)=-1$, we obtain
\begin{equation}\label{eq:null_norm_whirl}
2xz+y^2=-1.
\end{equation}
Now write
\[
\Phi T=aT+bN+cB.
\]
Since $\Phi$ is skew-symmetric with respect to $g$, we have $g(\Phi T,T)=0.$ Using \eqref{eq:cartan_metric_nonexist}, this gives $c=0.$ Thus $\Phi T=aT+bN.$ Next, using the standard identity $\eta\circ\Phi=0$ \cite{blair2006contact}, we have
\[
0=\eta(\Phi T)=-g(\Phi T,\xi)=-(az+by).
\]
Using $y=\rho z$ in the above equality, we get
\[
(a+\rho b)z=0.
\]
If $z=0$, then $y=0$, and \eqref{eq:null_norm_whirl} gives $0=-1$, a contradiction. Hence $z\neq0$, and therefore
\[
a=-\rho b.
\]
Finally, using  the Sasakian identity \eqref{eq:gdeta}, we have
\[
g(\Phi T,\Phi T)=g(T,T)+\eta(T)^2=z^2.
\]
On the other hand, since $\Phi T=aT+bN$ and \eqref{eq:cartan_metric_nonexist} holds, we obtain
\[
g(\Phi T,\Phi T)=b^2.
\]
Thus $b^2=z^2,$ so there exists $\delta\in\{\pm1\}$ such that $b=\delta z,$ $a=-\delta\rho z.$ Therefore
\begin{equation}\label{eq:phiT_null_whirl}
\Phi T=\delta z\,(N-\rho T), \quad \text{where}\quad  \delta=\pm1.
\end{equation}
On the other hand, differentiating $\xi=xT+yN+zB$ along $\alpha$ and using \eqref{eq:cartan_frenet_nonexist}, we get
\[
\nabla_T\xi
=
(x'-\tau y)T
+
(x+y'+\tau z)N
+
(z'-y)B.
\]
Since $M$ is Sasakian, $\nabla_T\xi=\Phi T$, and comparing with \eqref{eq:phiT_null_whirl} yields
\begin{align}
x'-\rho\tau z &= -\delta\rho z,\label{eq:null_whirl_sys1}\\
x+y'+\tau z &= \delta z,\label{eq:null_whirl_sys2}\\
z'-y &= 0.\label{eq:null_whirl_sys3}
\end{align}
Using $y=\rho z$, equation \eqref{eq:null_whirl_sys3} becomes $z'=\rho z,$ hence 
\begin{equation}\label{eq:null_whirl_z}
z(s)=ke^{\rho s},
\qquad k\neq0.
\end{equation}
Furthermore, substituting $y=\rho z$ into \eqref{eq:null_norm_whirl}, we obtain
\begin{equation}\label{eq:null_whirl_x}
x= \frac{-1-\rho^2 z^2}{2z}.
\end{equation}
Differentiating \eqref{eq:null_whirl_x} and using $z'=\rho z$, we obtain
$$
x'=\frac{\rho}{2z}-\frac{\rho^3}{2}z.
$$
Substituting this into \eqref{eq:null_whirl_sys1} and using $z=ke^{\rho s}$ yields 
\begin{equation}\label{eq:null_whirl_tau}
\tau(s)=\delta+\frac{1}{2k^2}e^{-2\rho s}-\frac{\rho^2}{2}.
\end{equation}
The preceding computations yield the following structural result.
\begin{proposition}\label{prop:null_whirl_necessary}
Let $\alpha$ be a non-geodesic null contact whirl curve in a Lorentzian Sasakian $3$--manifold. Then, writing $\rho=1/\varrho$, there exist constants $k\neq0$ and $\delta=\pm 1$ such that
\begin{align*}
\eta(T)&=-ke^{\rho s},\\
\eta(N)&=-\rho ke^{\rho s},\\
\Phi T&=\delta ke^{\rho s}(N-\rho T),\\
\tau(s)&=\delta-\frac{\rho^2}{2}+\frac{1}{2k^2}e^{-2\rho s}.
\end{align*}
In particular, the null whirl condition imposes strong restrictions on the Reeb components and on the torsion.
\end{proposition}
Consequently, along any null contact whirl curve, the Reeb vector field admits the explicit decomposition
$$
\xi=\left(-\frac{1}{2k}e^{-\rho s}-\frac{k\rho^2 }{2}e^{\rho s}
\right)T+\rho k e^{\rho s}N+k e^{\rho s}B\quad\text{for }\quad k\neq 0.
$$
Although the Reeb vector field is part of the ambient Sasakian structure, Proposition \ref{prop:null_whirl_necessary} shows that its expression with respect to the Cartan frame of a null whirl curve is rigidly constrained.
\begin{remark}
Proposition \ref{prop:null_whirl_necessary} shows that the null whirl condition imposes strong intrinsic restrictions on the Cartan frame, the Reeb components and the torsion of the curve. Furthermore, Proposition \ref{prop:null_whirl_necessary} may be regarded as the null analogue of Lee's structural result for null slant curves \cite[Theorem~2]{lee2019slant}. In the present setting, the constant contact-angle condition is replaced by the whirl relation coupling the Reeb components of the tangent and principal normal vectors.  
\end{remark}
\section{Explicit examples in the Sasakian Lorentzian Heisenberg group}\label{Sec:5}
We consider the Lorentzian Heisenberg space $\mathbb{H}_3^L$, which is a
three-dimensional Lie group with global coordinates $(x,y,z)\in\mathbb{R}^3$.
Following the convention used in \cite{lee2019slant}, we take the contact form
\begin{equation}\label{eq:etaHeisenberg}
\eta = dz + y\,dx - x\,dy .
\end{equation}
The contact distribution is given by $\ker\eta=\operatorname{span}\{X,Y\}, $
where 
\begin{equation}\label{eq:invariantFrame}
X=\partial_x-y\,\partial_z,\qquad
Y=\partial_y+x\,\partial_z,\qquad
\xi=\partial_z.
\end{equation}
Indeed, it is straightforward to see that  $\eta(X)=0=\eta(Y)$, and $\eta(\xi)=1.$ The Lorentzian metric is
\begin{equation}\label{eq:metricHeisenberg}
g=dx^2+dy^2-\bigl(dz+y\,dx-x\,dy\bigr)^2.
\end{equation}
With respect to this metric, the frame $\{X,Y,\xi\}$ is orthonormal of signature $(+,+,-)$, i.e.,
$$
g(X,X)=1,\qquad g(Y,Y)=1,\qquad g(\xi,\xi)=-1,
$$
and all mixed products vanish. The nonzero Lie bracket is
$[X,Y]=2\xi.$ With respect to the frame $\{X,Y,\xi\}$, the Levi--Civita connection is
\begin{equation}\label{eq:LC_Heisenberg}
\begin{aligned}
\nabla_X X &=0, \qquad \nabla_Y Y=0, \qquad \nabla_\xi\xi=0, \qquad
\nabla_XY=\xi, \qquad \nabla_YX=-\xi, \\[4pt]
\nabla_Y\xi &=-X, \qquad \nabla_X\xi=Y, \qquad \nabla_\xi X=Y, \qquad
 \nabla_\xi Y=-X.
\end{aligned}
\end{equation}
The structure tensor is defined by
\begin{equation}\label{eq:sasakianHein}
\Phi X=Y,\qquad \Phi Y=-X,\qquad \Phi\xi=0.
\end{equation}
With this convention, $(\mathbb H_3^L,\Phi,\xi,\eta,g)$ is the standard Lorentzian Sasakian Heisenberg model. \\

Let $\alpha(s)=(x(s),y(s),z(s))$  be a regular curve in $\mathbb H_3^L$. Its velocity vector is
$$
\alpha'(s)=x'(s)\partial_x+y'(s)\partial_y+z'(s)\partial_z.
$$
Using \eqref{eq:invariantFrame}, we have $\partial_x=X+y\xi,$ $\partial_y=Y-x\xi,$ and $\partial_z=\xi.$ Therefore
\begin{equation}\label{eq:tangent_heins}
T=\alpha'=x'X+y'Y+\bigl(z'+yx'-xy'\bigr)\xi.
\end{equation}
Using the metric \eqref{eq:metricHeisenberg}, we obtain
\begin{equation}\label{eq:heis_norm}
g(T,T)=(x')^2+(y')^2-\left(z'+yx'-xy'\right)^2.
\end{equation}
Assume that $\alpha$ is parameterized by arc-length, so that
$g(T,T)=\varepsilon_1,$ and $\varepsilon_1\in\{1,-1\}.$ We denote by
$$
v:=z'+yx'-xy'
$$
the Reeb component of $T$. Then $\eta(T)=v.$  Using \eqref{eq:LC_Heisenberg} and \eqref{eq:tangent_heins}, we now compute the acceleration of $\alpha$. 
\[
\nabla_TT= x''X+y''Y+v'\xi +x'\nabla_T X+y'\nabla_TY+v\nabla_T\xi.
\]
From \eqref{eq:LC_Heisenberg}, we have
\begin{align*}
\nabla_T X &= x'\nabla_X X+y'\nabla_YX+v\nabla_\xi X = -y'\xi+vY,\\
\nabla_TY &= x'\nabla_XY+y'\nabla_YY+v\nabla_\xi Y
=x'\xi-vX,\\
\nabla_T\xi &= x'\nabla_X\xi+y'\nabla_Y\xi+v\nabla_\xi\xi=x'Y-y'X.    
\end{align*} 
Therefore,
\begin{equation}\label{eq:eta_acceleration_vector}
 \begin{aligned} 
\nabla_TT
&=x''X+y''Y+v'\xi+x'(-y'\xi+vY) +y'(x'\xi-vX)+v(x'Y-y'X)\\
&=\bigl(x''-2y'v\bigr)X+\bigl(y''+2x'v\bigr)Y+v'\xi. 
\end{aligned}    
\end{equation}
Consequently,
\begin{equation}\label{eq:eta_acceleration}
 \eta(\nabla_TT)=v'.   
\end{equation} 
Assume that $\alpha$ is a non-geodesic Frenet curve. Then, by definition of the principal normal vector field,
$$
\nabla_TT=\varepsilon_2\kappa\,N,
$$
If, in addition, $\alpha$ satisfies the contact whirl condition
\eqref{eq:whirlcurve},  $\eta(T)=\varrho\,\eta(N),$ where   $\varrho\neq0,$ then applying $\eta$ to the Frenet equation and using the whirl relation, $\eta(N)=\frac{1}{\varrho}\eta(T),$  we obtain
$$
\eta(\nabla_TT)=\varepsilon_2\kappa\,\eta(N)=
\frac{\varepsilon_2\kappa}{\varrho}\,\eta(T).
$$
Therefore, for a non-geodesic Frenet contact whirl curve,
\begin{equation}\label{eq:whirl_definition}
\eta(\nabla_TT)=\lambda(s)\eta(T),
\end{equation}
where
$ 
\lambda(s)=\frac{\varepsilon_2\kappa(s)}{\varrho}.
$ 
Hence, by \eqref{eq:eta_acceleration} and $\eta(T)=v$, equation
\eqref{eq:whirl_definition} becomes
\begin{equation}\label{eq:whirl_ode}
v'=\lambda(s)\,v.
\end{equation}
Equation \eqref{eq:whirl_ode} can be used as a construction principle for non-Legendre contact whirl curves.
\subsection{A construction method by quadratures}
\label{sec:construction_principle}
Write the unit tangent vector in the form
$$
T=r(s)\cos\beta(s)\,X+r(s)\sin\beta(s)\,Y+v(s)\xi,
$$
then $\eta(T)=v.$ Also, from $g(T,T)=\varepsilon_1$ we have
\begin{equation}\label{eq:e_1condition}
r(s)^2-v(s)^2=\varepsilon_1, \qquad  \varepsilon_1\in\{1,-1\}.    
\end{equation} 
Assume that $v\neq0$ and $r\neq0$ on the interval under consideration. Substituting $x'=r\cos\beta$ and $y'=r\sin\beta$ into \eqref{eq:eta_acceleration_vector}, we obtain
\begin{equation}\label{eq:polar_acceleration}
\nabla_TT= \bigl(r'\cos\beta-r(\beta'+2v)\sin\beta\bigr)X
+\bigl(r'\sin\beta+r(\beta'+2v)\cos\beta\bigr)Y +v'\xi.
\end{equation}
Consequently,
$$
g(\nabla_TT,\nabla_TT)=(r')^2+r^2(\beta'+2v)^2-(v')^2.
$$
Let $\varrho\neq0$ and $\lambda_0\neq0$ be constants, and choose $\varepsilon_2\in\{1,-1\}$  so that
\begin{equation}\label{eq:sign_condition}
\varepsilon_2\varrho\lambda_0>0.
\end{equation}
Assume that $v'=\lambda_0 v.$ Define $Q(s)=\varepsilon_2\varrho^2\lambda_0^2+(v')^2-(r')^2.$ 
On any interval where $Q(s)\geq0,$  choose $\beta$ as a solution of
$$
\beta' = -2v \pm \frac{\sqrt{Q}}{r}.
$$
Equivalently, choose $\beta$ such that
\begin{equation}\label{eq:beta_construction}
r^2(\beta'+2v)^2 = \varepsilon_2\varrho^2\lambda_0^2 + (v')^2-(r')^2.
\end{equation}
With this choice, the previous expression for the acceleration gives
$$
g(\nabla_TT,\nabla_TT)
=
\varepsilon_2\varrho^2\lambda_0^2.
$$
Hence the curve is non-geodesic, and its curvature satisfies
$\kappa^2=\varrho^2\lambda_0^2.$   By the sign condition \eqref{eq:sign_condition}, we may write
$$
\kappa=\varepsilon_2\varrho\lambda_0.
$$
In particular, $\kappa$ is constant and $\lambda_0=\frac{\varepsilon_2\kappa}{\varrho}.$  On the other hand, from \eqref{eq:polar_acceleration},
$$
\eta(\nabla_TT)=v'=\lambda_0 v=\lambda_0\eta(T).
$$
Therefore,
$$
\eta(\nabla_TT)
=
\frac{\varepsilon_2\kappa}{\varrho}\eta(T).
$$
Since $\nabla_TT=\varepsilon_2\kappa N,$ we also have $\eta(\nabla_TT)=\varepsilon_2\kappa\,\eta(N).$  As $\kappa\neq0$, it follows that
$$
\eta(T)=\varrho\,\eta(N).
$$
Thus the resulting curve satisfies the defining contact whirl relation
\eqref{eq:whirlcurve}. Finally, a curve realizing this tangent field is obtained by solving
$$
\begin{cases}
 x'=r\cos\beta,\\
y'=r\sin\beta,\\
z'=v-yx'+xy'.   
\end{cases} 
$$ 
Since $v=\eta(T)\neq0$, the constructed contact whirl curve is non-Legendre. 
\subsection{Explicit Examples}\label{sec:examples}
\subsubsection{A non-Legendre contact whirl curve obtained from \eqref{eq:whirl_ode}}
We construct a spacelike non-Legendre contact whirl curve by imposing
\eqref{eq:whirl_ode} with constant factor. Let
$  \varepsilon_1= \varepsilon_2= \varrho=\lambda_0=1. $   Choose
$$
v(s)=e^s,
\qquad
r(s)=\sqrt{1+e^{2s}}.
$$
Then \eqref{eq:e_1condition} holds,
so the corresponding tangent field is unit spacelike. Moreover,
$$
v'(s)=e^s=v(s),
$$
and therefore \eqref{eq:whirl_ode} holds with $\lambda(s)=\lambda_0=1.$   Now define $\beta$ as a solution of \eqref{eq:beta_construction}. Since here
$$
(r')^2=\frac{e^{4s}}{1+e^{2s}}, \qquad (v')^2=e^{2s},
$$
equation \eqref{eq:beta_construction} becomes
$$
r^2(\beta'+2v)^2 = 1+\frac{e^{2s}}{1+e^{2s}}.
$$
Equivalently, $\beta'=-2e^s\pm\frac{\sqrt{1+2e^{2s}}}{1+e^{2s}}.$ 
Hence
$$
\beta(s)=\beta_0+\int_{s_0}^{s}
\left(
-2e^u
\pm
\frac{\sqrt{1+2e^{2u}}}{1+e^{2u}}
\right)\,du.
$$
With this choice, \eqref{eq:beta_construction} yields $g(\nabla_TT,\nabla_TT)=1.$  Therefore the curve is non-geodesic and has constant curvature $\kappa=1.$  Since also
$$
\lambda_0=1=\frac{\varepsilon_2\kappa}{\varrho},
$$
the construction satisfies the defining whirl relation \eqref{eq:whirlcurve}. Finally, a curve realizing this tangent field is obtained by solving $x'=r\cos\beta,$ $y'=r\sin\beta,$ and $
z'=v-yx'+xy'.$  That is,
\begin{align*}
x(s) &=x_0+\int_{s_0}^{s}\sqrt{1+e^{2u}}\cos\beta(u)\,du,\\  
y(s) &=y_0+\int_{s_0}^{s}\sqrt{1+e^{2u}}\sin\beta(u)\,du,\\
z(s) &=z_0+\int_{s_0}^{s}\Bigl(e^u-y(u)x'(u)+x(u)y'(u)\Bigr)\,du.   
\end{align*} 
Since $\eta(T)=v=e^s\neq0,$ the resulting contact whirl curve is non-Legendre. 
\subsubsection{A horizontal helicoidal Legendre contact whirl curve }
Let $R>0$ and $\omega\neq0$, and consider
$$
\alpha(s)= \left(R\cos(\omega s),\,R\sin(\omega s),\, R^2\omega s \right).
$$
Assume moreover that the curve is parametrized by arc-length, namely
$R^2\omega^2=1.$ Then
$$
x'=-R\omega\sin(\omega s),\qquad
y'=R\omega\cos(\omega s),\qquad
z'=R^2\omega.
$$
A direct computation gives $yx'-xy'=-R^2\omega.$ Hence
$$
v=z'+yx'-xy'=R^2\omega-R^2\omega=0.
$$
Therefore $\eta(T)=0,$ so $\alpha$ is a Legendre curve. Using \eqref{eq:eta_acceleration_vector} and the fact that $v=0$, we obtain
$$
\nabla_TT=
-R\omega^2\cos(\omega s)\,X
-R\omega^2\sin(\omega s)\,Y.
$$
Since $R>0$ and $\omega\neq0$, this vector never vanishes. Hence $\nabla_TT\neq0,$ so the curve is non-geodesic. Moreover, $\nabla_TT$ has no Reeb component, so $\eta(\nabla_TT)=0.$ Since $\nabla_TT=\varepsilon_2\kappa N$ and $\kappa\neq0,$ it follows that $\eta(N)=0.$ Therefore,
$$
\eta(T)=0=\varrho\,\eta(N)
$$
for any nonzero constant $\varrho$, and hence $\alpha$ satisfies the contact whirl relation \eqref{eq:whirlcurve}. Geometrically, this family consists of horizontal helices winding around the Reeb axis.
%
%
%
\bibliographystyle{alpha}
\bibliography{sample}
\end{document}